\documentclass[11pt]{amsart}
\usepackage{amsmath}
\usepackage{latexsym}
\usepackage{amsfonts}
\usepackage{amssymb}

\newcommand{\beqa}{\begin{eqnarray*}}
\newcommand{\eeqa}{\end{eqnarray*}}
\newcommand{\beqn}{\begin{eqnarray}}
\newcommand{\eeqn}{\end{eqnarray}}

\newcommand{\iy}{\infty}
\newcommand{\lt}{\left}
\newcommand{\rt}{\right}

\newcommand{\bQ}{\mathbb Q}

\newcommand{\C}{\mathbb C}

\newcommand{\R}{\mathbb R}

\newcommand{\N}{\mathbb N}

\newcommand{\Ha}{\mathbb H}

\newcommand{\mcD}{\mathcal D}

\newcommand{\mcE}{\mathcal E}

\newcommand{\mcS}{\mathcal S}

\newcommand{\mfM}{\mathfrak M}

\newcommand{\f}{\frac}
\newcommand{\tf}{\tfrac}

\newcommand{\al}{\alpha}

\newcommand{\G}{\Gamma}

\newcommand{\e}{\varepsilon}
\newcommand{\ph}{\phi}

\newcommand{\De}{\Delta}
\newcommand{\la}{\lambda}

\newcounter{cnt1}
\newcounter{cnt2}
\newcounter{cnt3}
\newcommand{\blr}{\begin{list}{$($\roman{cnt1}$)$}
 {\usecounter{cnt1} \setlength{\topsep}{0pt}
 \setlength{\itemsep}{0pt}}}
\newcommand{\bla}{\begin{list}{$($\alph{cnt2}$)$}
 {\usecounter{cnt2} \setlength{\topsep}{0pt}
 \setlength{\itemsep}{0pt}}}
\newcommand{\bln}{\begin{list}{$($\arabic{cnt3}$)$}
 {\usecounter{cnt3} \setlength{\topsep}{0pt}
 \setlength{\itemsep}{0pt}}}
\newcommand{\el}{\end{list}}
\newtheorem{thm}{Theorem}[section]
\newtheorem{lem}[thm]{Lemma}

\newtheorem{Def}[thm]{Definition}

\newtheorem{rem}[thm]{Remark}
\newcommand{\Rem}{\begin{rem} \rm}
\newcommand{\bdfn}{\begin{Def} \rm}
\newcommand{\edfn}{\end{Def}}

\newcommand{\ba}{\begin{array}}
\newcommand{\ea}{\end{array}}

\numberwithin{equation}{section}
\date{}
\begin{document}
\title{\bf The Jones Strong Distribution Banach Spaces}
\author[Gill]{Tepper L. Gill}
\address[Tepper L. Gill]{ Department of Mathematics, Physics and E\&CE, Howard University\\
Washington DC 20059 \\ USA, {\it E-mail~:} {\tt tgill@howard.edu}}
% [-1ex] \normalsize \sc Howard University\\
%\abstract{ write abstract here}
\date{}
%thispagestyle{empty}
\subjclass{Primary (45) Secondary(46) }
\maketitle
\begin{abstract}  In this note, we introduce a new class of separable Banach spaces, ${SD^p}[{\mathbb{R}^n}],\;1 \leqslant p \leqslant \infty$, which contain each $L^p$-space as a dense continuous and compact embedding.  They also contain the nonabsolutely integrable functions and the space of test functions ${\mathcal{D}}[{\mathbb{R}^n}]$, as dense continuous embeddings. These spaces have the remarkable property that, for any multi-index $\al, \; \left\| {{D^\alpha }{\mathbf{u}}} \right\|_{SD} = \left\| {\mathbf{u}} \right\|_{SD}$, where $D$ is the distributional derivative. We call them Jones strong distribution Banach spaces because of the crucial role played by two special functions introduced in his book (see \cite{J}, page 249).  After constructing the spaces, we discuss their basic properties and their relationship to ${\mathcal{D}}[{\mathbb{R}^n}]$ and ${\mathcal{D'}}[{\mathbb{R}^n}]$. As an application, we obtain new a priori bounds for the  Navier-Stokes equation. 
\end{abstract}
\section*{Introduction}
The theory of distributions is based on the action of linear functionals on a space of test functions. In the original approach of Schwartz \cite{SC}, both the test functions and the linear functionals have a natural topological vector space structure, which is not normable. For those interested in applications, this is an inconvenience, requiring additional study and effort.  Thus, in most applied contexts, the restricted class of Banach spaces due to Sobolev have proved useful (see Leoni \cite{GL}).  In this case, the base space of functions are of Lebesgue type, which also have some limitations.  The study of problems in the foundations of relativistic quantum theory, path integrals and nonlinear analysis have led to the need for a Banach space structure for nonabsolutely integrable functions.  Recent research has uncovered a new class of separable Banach spaces, which are natural for the class of nonabsolutely integrable functions and have provided the first rigorous foundation for the Feynman path integral formulation of quantum mechanics (see \cite{GZ}).  The purpose of this note is to introduce a another class of Banach spaces which contain the nonabsolutely integrable functions, but also contains the Schwartz test function space as dense and continuous embeddings.
\section{The Jones Spaces}
We begin with the construction of a special class of functions in $\C_c^{\iy}[\R^n]$ (see Jones, \cite{J} page 249). 
\subsection{The remarkable Jones functions}  
\begin{Def} For $x \in \R, \ 0 \le y < \iy$ and $1<a< \iy$, define the Jone's functions  $g(x, y), \ h(x)$ by:
\[
  g(x,y) = \exp \left\{ { - y^a e^{iax} } \right\},
\]
\[ 
 h(x)=\left\{
\begin{array}{ll}
\displaystyle{\int_0^\infty g(x,y) dy}, & x\in  (-\frac{\pi}{2a},\frac{\pi}{2a}) \\
~&~\\
0, & \mbox{otherwise.}
\end{array}
\right.\]
\end{Def}
The following properties of $g$ are easy to check:
\begin{enumerate}
\item 
\[
\frac{{\partial g(x,y)}}
{{\partial x}} =  - iay^a e^{iax} g(x,y),
\]
\item
\[
\frac{{\partial g(x,y)}}
{{\partial y}} =  - ay^{a - 1} e^{iax} g(x,y),
\]
so that
\item
\[
iy\frac{{\partial g(x,y)}}
{{\partial y}} = \frac{{\partial g(x,y)}}
{{\partial x}}.
\]
\end{enumerate}
It is also easy to see that $h(x) \in {{L}}^1[- \tf{\pi}{2a}, \tf{\pi}{2a}]$ and,
\beqn
\frac{{dh(x)}}
{{dx}} = \int_0^\infty  {\frac{{\partial g(x,y)}}
{{\partial x}}dy}  = \int_0^\infty  {iy\frac{{\partial g(x,y)}}
{{\partial y}}dy}.
\eeqn
Integration by parts in the last expression in (1.1) shows that $h'(x) =  - ih(x)$, so that $
h(x) = h(0)e^{ - ix}$ for $x \in (- \tf{\pi}{2a}, \tf{\pi}{2a})$.  Since $
h(0) = \int_0^\infty  {\exp \{  - y^a \} dy}$, an additional integration by parts shows that $h(0)= \G(\tf{1}{a} +1)$. For each $k \in \N$ let $a=a_k= 3\times 2^{k-1},\; h(x)=h_k(x), \ x \in (- \tf{\pi}{2a_k}, \tf{\pi}{2a_k})$ and set $\e_k=\tf{\pi}{4a_k}$. 

  Let $\bQ$ be the set of rational numbers in $\R$ and for each $x^i \in \bQ$,  define
\beqa
f_k^{i} (x) = f_k(x-x^i)=\left\{ {\begin{array}{*{20}c}
   {c_k \exp \left\{ {\frac{{\varepsilon _k^2 }}
{{\left| {x - x^i} \right|^2  - \varepsilon _k^2 }}} \right\},\quad \left| {x - x^i } \right| < \varepsilon _k ,}  \\
   {0,\quad \quad \quad \quad \quad \quad \quad \quad \quad \quad \quad \left| {x - x^i } \right| \geqslant \varepsilon _k ,}  \\

 \end{array} } \right.
\eeqa
where $c_k$ is the standard normalizing constant.  It is clear that the support, ${\rm{spt}}(f_k^{i}) \subset [-\e_k, \e_k]=[-\tf{\pi}{4a_k}, \tf{\pi}{4a_k}]=I_k^i$.

Now set $\chi _k^i (x) =  (f_k^i  * h_k)(x)$, so that ${\rm{spt}}(\chi _k^i ) \subset [-\tf{\pi}{2^{k+1}}, \tf{\pi}{2^{k+1}}]$.  For $x \in {\rm{spt}}(\chi _k^i )$, we can also write $\chi _k^i (x)=\chi _k(x-x^i)$ as:
\[
\begin{gathered}
\chi _k^i (x)=  \int_{ I_k^i }  {{f_k}\left[ {\left( {x - {x^i}} \right) - z} \right]{h_k}(z)dz}  \hfill \\
   = \int_{ I_k^i }  {{h_k}\left[ {\left( {x - {x^i}} \right) - z} \right]{f_k}(z)dz}  = {e^{-i\left( {x - {x^i}} \right)}}\int_{ I_k^i }  {{e^{iz}}{f_k}(z)dz} . \hfill \\ 
\end{gathered} 
\]
Thus, if $\alpha_{k,i}  = \int_{I_k^i}  {e^{iz} f_k^i (z)dz} $, we can now define:
\[
\xi _k^i (x) =\al_{k,i}^{-1} \chi_k^i (-x)= \left\{ {\begin{array}{*{20}c}
   {\f{1}{n} {e^{i(x-x^i)}} ,\;x \in {I_k^i} }  \\
   {0,\;\,x \notin {I_k^i} },  \\

 \end{array} } \right.
\]
so that $\lt|\xi _k^i (x)\rt| < \tf{1}{n}$.
\subsection{The Construction}
To construct our space on $\R^n$, let $\bQ^n$ be the set of all vectors ${\mathbf{x}}$ in $ {\mathbb{R}}^n$, such that for each $j$, the component $x_j$ is rational.  Since this is a countable dense set in ${\mathbb{R}}^n $, we can arrange it as $\mathbb{Q}^n  = \left\{ {{\mathbf{x}}^1, {\mathbf{x}}^2, {\mathbf{x}}^3,  \cdots } \right\}$.  For each $k$ and $i$, let ${\mathbf{B}}_k ({\mathbf{x}}^i ) $ be the closed cube centered at ${\mathbf{x}}^i$ with edge $\tfrac{\pi }{{a_k }}$ and diagonal of length $r_k=\tfrac{\pi }{{a_k}}\sqrt n$.  

We choose the natural order which maps $\mathbb{N} \times \mathbb{N}$ bijectively to $\mathbb{N}$:
\[
\{(1,1), \ (2,1), \ (1,2), \ (1,3), \  (2,2), \  (3,1), \ (3,2), \ (2,3), \  \cdots \}
\]
 and let $\left\{ {{\mathbf{B}}_{m} ,\;m \in \mathbb{N}}\right\}$ be the set of closed cubes 
${\mathbf{B}}_k ({\mathbf{x}}^i )$ with $(k,i) \in \mathbb{N} \times \mathbb{N}$ and ${\bf{x}^i} \in\mathbb{Q}^n $. For each ${\bf{x}} \in {\bf{B}}_m, \; {\bf{x}}=(x_1, x_2, \dots, x_n)$, we define  $\mcE_k ({\mathbf{x}})$ by :
\[
\begin{gathered}
  { {\mathcal{E}}_k}({\mathbf{x}}) = \left( {\xi _k^i({x_1}),\xi _k^i({x_2}) \ldots \xi _k^i({x_n})} \right) \; with \hfill \\
  \left| {{ {\mathcal{E}}_k}({\mathbf{x}})} \right| < {1},\;\;{\mathbf{x}} \in \prod\nolimits_{j = 1}^n {I_k^i} {\text{ and  }}{ {\mathcal{E}}_k}({\mathbf{x}}) = 0,\;{\text{ }}{\mathbf{x}} \notin \prod\nolimits_{j = 1}^n {I_k^i} . \hfill \\ 
\end{gathered} 
\]
It is easy to see that ${\mathcal{E}}_k ({\mathbf{x}})$ is in ${{{L}}^p [{\mathbb{R}}^n ]^n }={{\bf{L}}^p [{\mathbb{R}}^n ] }$ for $1 \le p \le\infty$.  Define $F_{k} (\; \cdot \;)$ on $
{{\bf{L}}^p [{\mathbb{R}}^n ] }$ by
 \beqa
F_{k} (f) = \int_{{\mathbb{R}}^n } {{\mathcal{E}}_{k} ({\mathbf{x}}) \cdot f({\mathbf{x}})d\la_n({\mathbf{x}})}. 
\eeqa

It is clear that $F_{k} (\; \cdot \;)$ is a bounded linear functional on ${{\bf{L}}^p [{\mathbb{R}}^n ]} $ for each ${k}$ with $ \left\| {F_{k} } \right\|   \le 1$.  Furthermore,  if $F_k (f) = 0$ for all ${k}$, $f = 0$ so that $\left\{ {F_{k} } \right\}$ is a fundamental sequence of functionals  on ${{\bf{L}}^p [{\mathbb{R}}^n ]} $ for $1 \le p \le \infty$.

Set ${t_{k}}= \tfrac{1}{{2^k }} $ so that ${\sum\nolimits_{k = 1}^\infty  {t_k}}=1$  and  define a  inner product $\left( {\; \cdot \;} \right) $ on ${{\bf{L}}^1 [{\mathbb{R}}^n ]} $ by
\beqa
 \left( {f,g} \right) = \sum\nolimits_{k = 1}^\infty  {t_k } \left[ {\int_{\mathbb{R}^n } {{\mathcal{E}}_k ({\mathbf{x}})\cdot f({\mathbf{x}})d\la_n({\mathbf{x}})} } \right] \overline{\left[ {\int_{\mathbb{R}^n } {{\mathcal{E}}_k ({\mathbf{y}}) \cdot g({\mathbf{y}})d\la_n({\mathbf{y}})} } \right]}.    
\eeqa
 The completion of ${{\bf{L}}^1 [{\mathbb{R}}^n ]} $ with the above inner product is a Hilbert space, which we denote as ${{SD}}^2 [{\mathbb{R}}^n ] $.   
\begin{thm} For each $p,\;1 \leqslant p \leqslant \infty$, we have:
\begin{enumerate}
\item The space ${SD}^2 [{\R}^n ] \supset {\bf{L}}^p [\R^n ]$ as a continuous, dense and \underline{compact}  embedding.
\item The space ${{SD}}^2 [{\R}^n ] \supset \mathfrak{M}[{\R}^n ]$, the space of finitely additive measures on $\R^n$, as a continuous dense and \underline{compact}  embedding.
\end{enumerate}
\end{thm}
\begin{proof}
Since ${SD}^2 [{\mathbb{R}}^n ]$ contains ${{L}}^1 [{\mathbb{R}}^n ]$ densely, to prove (1), we need only show that ${{L}}^q [{\mathbb{R}}^n ]  \subset {SD}^2 [{\mathbb{R}}^n ]$ for $q \ne 1$.  Let $f \in {{L}}^q [{\mathbb{R}}^n ]$ and $q < \infty $. By construction, for every $j$,  $\left| { {\mathcal{E}}_j({\mathbf{x}})} \right| < \tf{1}{\sqrt{n}}$ so that there is a constant $C=C(q)$, with ${\left| { {\mathcal{E}}_j({\mathbf{x}})} \right|^q} \leqslant  C{\left| { {\mathcal{E}}_j({\mathbf{x}})} \right|}$.  It follows that:  
\[
\begin{gathered}
 \left\| f \right\|_{{SD}^2}  = \left[ {\sum\nolimits_{k = 1}^\infty  {t_k } \left| {\int_{{\mathbb{R}}^n } {{\mathcal{E}}_k ({\mathbf{x}})f({\mathbf{x}})d\la_{n}({\mathbf{x}})} } \right|^{\frac{{2q}}{q}} } \right]^{1/2}  \hfill \\
{\text{       }} \leqslant C\left[ {\sum\nolimits_{k = 1}^\infty  {t_k } \left( {\int_{{\mathbb{R}}^n } {{\mathcal{E}}_k ({\mathbf{x}})\left| {f({\mathbf{x}})} \right|^q d\la_{n}({\mathbf{x}})} } \right)^{\frac{2}{q}} } \right]^{1/2}  \hfill \\
{\text{      }} \leqslant C\sup _k \left( {\int_{{\mathbb{R}}^n } {{\mathcal{E}}_k ({\mathbf{x}})\left| {f({\mathbf{x}})} \right|^q d\la_{n}({\mathbf{x}})} } \right)^{\frac{1}
{q}}  \leqslant C\left\| f \right\|_q . \hfill \\ 
\end{gathered} 
\]
Hence, $f \in{SD}^2 [{\mathbb{R}}^n ] $.  For $q = \infty $, first note that $ vol({\mathbf{B}}_k )^2 \le \left[ {\frac{1}
{{2\sqrt n }}} \right]^{2n}$, so we have 
\[
\begin{gathered}
  \left\| f \right\|_{{SD}^2}  = \left[ {\sum\nolimits_{k = 1}^\infty  {t_k } \left| {\int_{{\mathbb{R}}^n } {{\mathcal{E}}_k ({\mathbf{x}})f({\mathbf{x}})d\la_{n}({\mathbf{x}})} } \right|^2 } \right]^{1/2}  \hfill \\
  {\text{       }} \leqslant \left[ {\left[ {\sum\nolimits_{k = 1}^\infty  {t_k [vol({\mathbf{B}}_k )]^2 } } \right][ess\sup \left| f \right|]^2 } \right]^{1/2}  \leqslant {\left[ {\frac{1}
{{2\sqrt n }}} \right]^{n}}\left\| f \right\|_\infty  . \hfill \\ 
\end{gathered} 
\]
Thus $f \in{SD}^2 [{\mathbb{R}}^n ] $, and ${\bf{L}}^\infty  [{\mathbb{R}}^n ] \subset{SD}^2 [{\mathbb{R}}^n ]$.
To prove compactness, suppose $\{f_n \}$ is any weakly convergent sequence in ${\bf{L}}^p, \; 1 \le p \le \iy$ with limit $f$.  Since ${\mathcal{E}}_k \in {\bf{L}}^q, \; 1/p+ 1/q=1$,   
\[
\int_{\mathbb{R}^n } { {\mathcal{E}}_k ({\mathbf{x}}) \cdot \left[ {f_n ({\mathbf{x}}) - f({\mathbf{x}})} \right]d\la_{n}({\mathbf{x}})}  \to 0
\]
for each $k$.  It follows that $\{f_n \}$ converges strongly to $f$ in ${SD}^2$.

To prove (2), we note that $\mfM[\R^n]={\bf{L}}^1[\R^n]^{**} \subset{SD}^2[\R^n] $.
\end{proof}
\begin{Def}We call ${{SD}}^2 [{\mathbb{R}}^n ] $ the Jones strong distribution Hilbert space on $\R^n$.
\end{Def}
In order to justify our definition, let $\al$ be a multi-index of nonnegative integers, $\al = (\al_1, \ \al_2, \ \cdots \ \al_k)$, with $\left| \alpha  \right| = \sum\nolimits_{j = 1}^k {\alpha _j } $.   If $D$ denotes the standard partial differential operator, let 
\[
D^{\al}=D^{\al_1}D^{\al_2} \cdots D^{\al_k}.
\]
\begin{thm} Let $\mcD[\R^n]$  be $\C_c^\iy[\R^n]$ equipped with the standard locally convex topology (test functions).
\begin{enumerate}
\item If $\phi_n \to \phi$ in $\mcD[\R^n]$, then $\phi_n \to \phi$ in the norm topology of  $ SD^2[\R^n]$, so that $\mcD[\R^n] \subset {{SD}}^2[{\mathbb{R}}^n ]$ as a continuous dense embedding.
\item If $T \in \mcD'[\R^n]$, then $T \in {SD^2[\R^n]}'$, so that $\mcD'[\R^n] \subset {{SD}^2[{\mathbb{R}}^n ]}'$ as a continuous dense embedding.
\item For any $f, \ g   \in SD^2[\R^n]$ and any multi-index $\al, \ {\left( {{D^\alpha }f,g} \right)_{SD}}={(-i)^\alpha }{\left( {f,g}\right)_{SD}}$.
\item The functions  $\mcS[\R^n]$, of rapid decrease at infinity are contained in  $SD^2[\R^n]$ continuous embedding, so that $\mcS'[\R^n] \subset {{SD}^2[{\mathbb{R}}^n ]}'$.
\end{enumerate}
\end{thm}
\begin{proof} To prove (1), suppose that $\phi_n \to \phi$ in $\mcD[\R^n]$.  By definition, there exist a compact set $K \subset \R^n$, which is the support of $\ph_n -\phi$ and ${D^\alpha }{\phi_n}$ converges to ${D^\alpha }\phi $ uniformly on $K$ for every multi-index $\al$.  Let $\{ \mcE_{K_j} \}$ be the set of all $\mcE_j$, with support $K_j \subset K$.  If $\al$ is a multi-index, we have:
\[
\begin{gathered}
  \mathop {\lim }\limits_{k \to \infty } {\left\| {{D^\alpha }{\phi _k} - {D^\alpha }\phi } \right\|_{SD}} \hfill \\
= \mathop {\lim }\limits_{k \to \infty } {\left\{ {\sum\limits_{j = 1}^\infty  {{t_{{K_j}}}{{\left| {\int_{{\mathbb{R}^n}} {{{\mathcal{E}}_{K_j}}({\mathbf{x}}) \cdot \left[ {{D^\alpha }{\phi _k}({\mathbf{x}}) - {D^\alpha }\phi ({\mathbf{x}})} \right]d{\lambda _n}({\mathbf{x}})} } \right|}^2}} } \right\}^{1/2}} \hfill \\
   \leqslant M \mathop {\lim \sup }\limits_{k \to \infty, {\bf{x}} \in K } \left| {{D^\alpha }{\phi _k}({\mathbf{x}}) - {D^\alpha }\phi ({\mathbf{x}})} \right| = 0. \hfill \\ 
\end{gathered} 
\]
Thus, since $ \al$ is arbitrary, we see that $\mcD[\R^n] \subset SD^2[\R^n]$ as a continuous embedding.  Since 
$\C_c^\iy[\R^n]$ is dense in ${\bf{L}}^1[\R^n]$, $\mcD[\R^n]$ is dense in $SD^2[\R^n]$.
To prove (2) we note that, as $\mcD[\R^n]$ is a dense locally convex subspace of $SD^2[\R^n]$, by a corollary of the Hahn-Banach Theorem every continuous linear functional, $T$ defined on $\mcD[\R^n]$, can be extended to a continuous linear functional on $SD^2[\R^n]$.  By the Riesz representation Theorem, every continuous linear functional $T$, defined on $SD^2[\R^n]$ is of the form $T(f) = {\left( {f,g} \right)_{SD}}$, for some $g \in SD^2[\R^n]$.  Thus, $T \in {SD^2[\R^n]}'$ and, by the identification $T \leftrightarrow g$ for each $T$ in $\mcD'[\R^n]$, we can map  $\mcD'[\R^n]$ into $SD^2[\R^n]$ as a continuous dense embedding. 

To prove (3), recall that each ${\mathcal{E}}_k \in \C_c^{\iy}[{\mathbb{R}^n }]$ so that, for any $f \in SD^2[\R^n]$, 
\[ 
{\int_{\mathbb{R}^n } {{\mathcal{E}}_k ({\mathbf{x}}) \cdot D^{\al}{{f}}({\mathbf{x}})d\la_n({\mathbf{x}})} }= (-1)^{\lt|\al \rt|}{\int_{\mathbb{R}^n } D^{\al} {{\mathcal{E}}_k ({\mathbf{x}}) \cdot{{f}{}}({\mathbf{x}})d\la_n({\mathbf{x}})} }.
\]
An easy calculation shows that: 
\[ 
(-1)^{\lt|\al \rt|}{\int_{\mathbb{R}^n } {D^{\al}{\mathcal{E}}_k ({\mathbf{x}}) \cdot {{f}}({\mathbf{x}})d\la_n({\mathbf{x}})} }= (-i)^{\lt|\al \rt|}{\int_{\mathbb{R}^n } {{\mathcal{E}}_k ({\mathbf{x}}) \cdot{{f}{}}({\mathbf{x}})d\la_n({\mathbf{x}})} }.
\]
It now follows that, for any ${\bf g}  \in {{SD}}^2[{\mathbb{R}}^n ], \; (D^{\al}{{f}}, {\bf g})_{{{SD}}^2}= (-i)^{\lt|\al \rt|}({{f}{}},{\bf g})_{{{SD}}^2}$.
\end{proof} 
\subsection{Functions of Bounded Variation}
The integral due to Henstock \cite{HS} and Kurzweil \cite{KW} (HK-integral), is the easiest to learn and best known of those integrals that integrate nonabsolutely integrable functions, which also extend the Lebesgue integral.  A good discussion of all the standard types of integrals, with comparisons among themselves and the Lebesgue integral, can be found in Gordon \cite{GO}. 

The objective of this section is to show that every  HK-integrable function is in  ${{SD}}^2 [{\mathbb{R}}^n ] $.  To do this, we need to discuss a certain class of  functions of bounded variation.  For functions defined on $\R$, the definition of bounded variation is unique.  However, for functions on $\R^n, \; n \ge 2$, there are a number of distinct definitions.

The functions of  bounded variation in the sense of Cesari are well known to analysts working in partial differential equations  and geometric measure theory (see Leoni \cite{GL}).  
\begin{Def}A function $f \in L^1[\R^n]$ is said to be of bounded variation in the sense of Cesari or $f \in BV_c[\R^n]$, if $f \in L^1[\R^n]$ and each $i, \; 1 \le i \le n$, there exists a signed Radon measure $\mu_i$, such that  
\[
\int_{\mathbb{R}^n } {f({\mathbf{x}})\frac{{\partial \phi ({\mathbf{x}})}}
{{\partial x_i }}d\lambda _n ({\mathbf{x}})}  =  - \int_{\mathbb{R}^n } {\phi ({\mathbf{x}})d\mu _i ({\mathbf{x}})} ,
\]
for all $\phi \in \C_0^\iy[\R^n]$.
\end{Def} 

The functions of  bounded variation in the sense of Vitali \cite{TY1}, are well known to applied mathematicians and engineers with interest in error estimates associated with research in control theory, financial derivatives, high speed networks, robotics and in the calculation of certain integrals.   (See, for example \cite{KAA}, \cite{{NI}}, \cite{PT} or \cite{PTR} and references therein.)
For the general definition, see Yeong (\cite{TY1}, p. 175).  We present a definition that is sufficient for continuously  differentiable functions.
\begin{Def} A function $f$ with continuous partials is said to be of bounded variation in the sense of Vitali or $f \in BV_v[\R^n]$ if for all intervals $(a_i,b_i), \, 1 \le i \le n$,
\[
V(f)=\int_{a_1 }^{b_1 } { \cdots \int_{a_n }^{b_n } {\left| {\frac{{\partial ^n f({\mathbf{x}})}}
{{\partial x_1 \partial x_2  \cdots \partial x_n }}} \right|d\lambda _n ({\mathbf{x}})} }  < \infty. 
\]
\end{Def}
\begin{Def}We define $BV_{v,0}[\R^n]$ by: 
\[
BV_{v,0}[\R^n]= \{f({\bf x}) \in BV_v[\R^n]: f({\bf x}) \to 0, \; {\rm as} \; x_i \to -\iy \},
\]
where $x_i$ is any component of ${\bf x}$. 
\end{Def}
The following two theorems may be found in \cite{TY1}. (See p. 184 and 187, where the first is used to prove the second.)  If $[a_i, b_i] \subset \R$, we define $[{\bf a}, {\bf b}] \in \R^n$ by   $[{\bf a}, {\bf b}]=\prod_{k=1}^n{[a_i,b_i]}$.  (The notation $(RS)$ means Riemann-Stieltjes.)
\begin{thm}Let $f$ be HK-integrable on $\left[ {{\mathbf{a}},{\mathbf{b}}} \right]$ and let $g \in BV_{v,0}[\R^n]$, then $fg$ is 
HK-integrable and 
\[
(HK)\int_{[{\bf a}, {\bf b}] }{f({\bf x})g({\bf x})d\la_n({\bf x})}=(RS)\int_{[{\bf a}, {\bf b}] }{\lt\{(HK)\int_{[{\bf a}, {\bf x}] }f({\bf y})d\la_n({\bf y})\rt\}dg({\bf x})}
\] 
\end{thm}
\begin{thm}Let $f$ be HK-integrable on $\left[ {{\mathbf{a}},{\mathbf{b}}} \right]$ and let $g \in BV_{v,0}[\R^n]$, then $fg$ is 
HK-integrable and 
\[
\lt|(HK)\int_{[{\bf a}, {\bf b}] }{f({\bf x})g({\bf x})d\la_n({\bf x})}\rt|
\le \lt\|f\rt\|_DV_{[{\bf a}, {\bf b}] }(g)
\]
\end{thm}
\begin{lem}The space $HK[{\mathbb{R}}^n ]$, of all HK-integrable functions is contained in ${{SD}}^2 [{\mathbb{R}}^n ] $.
\end{lem}
\begin{proof}
Since each $\mcE_k({\bf x})$ is  continuous and differentiable, $\mcE_k({\bf x}) \in BV_{v,0}[\R^n]$, so that for  $f \in HK[{\mathbb{R}}^n ] $,
\[
\begin{gathered}
\left\| f \right\|_{{\bf{SD}}^2}^2  = \sum\nolimits_{k = 1}^\infty  {t_k } \left| {\int_{\mathbb{R}^n } {{\mathcal{E}}_k ({\mathbf{x}})\cdot f({\mathbf{x}})d{\mathbf{x}}} } \right|^2  \leqslant \sup _k \left| {\int_{\mathbb{R}^n } {{\mathcal{E}}_k ({\mathbf{x}})\cdot f({\mathbf{x}})d{\mathbf{x}}} } \right|^2  \hfill \\
\leqslant \left\| f \right\|_{HK}^2 [\sup _k V(\mcE_k)]^2 <\iy. \hfill \\
\end{gathered} 
\]
It follows that $f \in {{SD}}^2 [{\mathbb{R}}^n ]$. 
\end{proof}  
\subsection{{{The General Case,} ${{SD}}^p,\; 1 \le p \le \iy$ }} 
To construct ${SD}^p [{\mathbb{R}}^n ]$ for all $p$ and  for ${\bf f} \in {\bf{L}}^p$, define:
\[
\left\| {\bf f} \right\|_{{{SD}}^p }  = \left\{ {\begin{array}{*{20}c}
   {\left\{ {\sum\nolimits_{k = 1}^\infty  {t_k \left| {\int_{\mathbb{R}^n } { {\mathcal{E}}_k ({\mathbf{x}}) \cdot {\bf f}({\mathbf{x}})d\la_n({\mathbf{x}})} } \right|} ^p } \right\}^{1/p} ,1 \leqslant p < \infty},  \\
   {\sup _{k \geqslant 1} \left| {\int_{\mathbb{R}^n } { {\mathcal{E}}_k ({\mathbf{x}}) \cdot {\bf f}({\mathbf{x}})d\la_n({\mathbf{x}})} } \right|,p = \infty .}  \\

 \end{array} } \right.
\] 
It is easy to see that $\left\| \cdot \right\|_{{{SD}}^p }$ defines a norm on  ${\bf{L}}^p$.  If ${{{SD}}^p }$ is the completion of ${\bf{L}}^p$ with respect to this norm, we have:
\begin{thm} For each $q,\;1 \leqslant q \leqslant \infty,$
 ${SD}^p [{\mathbb{R}}^n ] \supset {\bf{L}}^q [{\mathbb{R}}^n ]$ as  dense continuous embeddings.
\end{thm}
\begin{proof} As in the previous theorem, by construction ${SD}^p [{\mathbb{R}}^n ]$
contains ${\bf{L}}^p [{\mathbb{R}}^n ]$ densely, so we need only show that 	
${SD}^p [{\mathbb{R}}^n ] \supset {\bf{L}}^q [{\mathbb{R}}^n ]$
for $q \ne p$.  First, suppose that $p< \infty$.  If ${\bf f} \in {\bf{L}}^q [{\mathbb{R}}^n ]$ and $q < \infty $, we have 
\[
\begin{gathered}
 \left\| {\bf f} \right\|_{{SD}^p}  = \left[ {\sum\nolimits_{k = 1}^\infty  {t_k } \left| {\int_{{\mathbb{R}}^n } {{\mathcal{E}}_k ({\mathbf{x}}) \cdot {\bf f}({\mathbf{x}})d\la_n({\mathbf{x}})} } \right|^{\frac{{qp}}{q}} } \right]^{1/p}  \hfill \\
{\text{       }} \leqslant \left[ {\sum\nolimits_{k = 1}^\infty  {t_k } \left( {\int_{{\mathbb{R}}^n } {\lt|{\mathcal{E}}_k ({\mathbf{x}})\rt|^q \left| {{\bf f}({\mathbf{x}})} \right|^q d\la_n({\mathbf{x}})} } \right)^{\frac{p}{q}} } \right]^{1/p}  \hfill \\
{\text{      }} \leqslant \sup _k \left( {\int_{{\mathbb{R}}^n } {\lt|{\mathcal{E}}_k ({\mathbf{x}})\rt|^q \left| {{\bf f}({\mathbf{x}})} \right|^q d\la_n({\mathbf{x}})} } \right)^{\frac{1}
{q}}  \leqslant \left\| f \right\|_q . \hfill \\ 
\end{gathered} 
\]
Hence, $f \in {SD}^p [{\mathbb{R}}^n ] $.  For $q = \infty $, we have 
\[
\begin{gathered}
  \left\| {\bf f} \right\|_{{SD}^p}  = \left[ {\sum\nolimits_{k = 1}^\infty  {t_k } \left| {\int_{{\mathbb{R}}^n } {{\mathcal{E}}_k ({\mathbf{x}}) \cdot {\bf f}({\mathbf{x}})d\la_n({\mathbf{x}})} } \right|^p } \right]^{1/p}  \hfill \\
  {\text{       }} \leqslant \left[ {\left[ {\sum\nolimits_{k = 1}^\infty  {t_k [vol({\mathbf{B}}_k )]^p } } \right][ess\sup \left| {\bf f} \right|]^p } \right]^{1/p}  \leqslant M\left\| f \right\|_\infty  . \hfill \\ 
\end{gathered} 
\]
Thus ${\bf f} \in {SD}^p [{\mathbb{R}}^n ] $, and ${\bf{L}}^\infty[{\mathbb{R}}^n ] \subset {SD}^p [{\mathbb{R}}^n ]$.   The case $p=\infty$ is obvious.
\end{proof}
\begin{thm} For ${SD}^p$, $1\leq p \leq \infty$, we have: 
\begin{enumerate}
\item If $p^{ - 1}  + q^{ - 1}  = 1$, then the dual space of ${SD}^p[{\mathbb{R}}^n ]$ is ${SD}^q[{\mathbb{R}}^n ]$.
\item The test function space ${\mcD}[{\mathbb{R}}^n ]$ is contain in ${SD}^{p}[{\mathbb{R}}^n ]$ as a continuous dense embedding.
\item If $K$ is a weakly compact subset of ${\bf{L}}^p[{\mathbb{R}}^n ]$, it is a strongly compact subset of ${SD}^{p}[{\mathbb{R}}^n ]$.
\item The space $ {SD}^{\infty}[{\mathbb{R}}^n ] \subset {SD}^p[{\mathbb{R}}^n ]$.
\end{enumerate}
\end{thm}
\section{Application}
Let ${[L^2({\mathbb{R}}^3)]^3}$ be the Hilbert space of square integrable functions on ${\mathbb {R}}^3$,  let ${\mathbb {H}}[ {\mathbb {R}}^3 ]$ be the completion of the  set of functions in $\left\{ {{\bf{u}} \in \mathbb {C}_0^\infty  [ {\mathbb {R}}^3 ]^3 \left. {} \right|\,\nabla  \cdot {\bf{u}} = 0} \right\}$ which vanish at infinity with respect to the inner product of ${[L^2({\mathbb{R}}^3)]^3 }$.  The classical Navier-Stokes initial-value problem (on $ \mathbb{R}^3 {\text{ and all }}T > 0$) is to find a  function ${\mathbf{u}}:[0,T] \times {\mathbb {R}}^3  \to \mathbb{R}^3$ and $p:[0,T] \times {\mathbb {R}}^3  \to \mathbb{R}$ such that
\beqn
\begin{gathered}
  \partial _t {\mathbf{u}} + ({\mathbf{u}} \cdot \nabla ){\mathbf{u}} - \nu \Delta {\mathbf{u}} + \nabla p = {\mathbf{f}}(t){\text{ in (}}0,T) \times {\mathbb {R}}^3 , \hfill \\
  {\text{                              }}\nabla  \cdot {\mathbf{u}} = 0{\text{ in (}}0,T) \times {\mathbb {R}}^3 {\text{ (in the weak sense),}} \hfill \\
    {\text{                              }}{\mathbf{u}}(0,{\mathbf{x}}) = {\mathbf{u}}_0 ({\mathbf{x}}){\text{ in }}{\mathbb {R}}^3. \hfill \\ 
\end{gathered} 
\eeqn
The equations describe the time evolution of the fluid velocity ${\mathbf{u}}({\mathbf{x}},t)$ and the pressure $p$ of an incompressible viscous homogeneous Newtonian fluid with constant viscosity coefficient $\nu $ in terms of a given initial velocity ${\mathbf{u}}_0 ({\mathbf{x}})$ and given external body forces ${\mathbf{f}}({\mathbf{x}},t)$.  

Let $\mathbb{P}$ be the (Leray) orthogonal projection of 
$(L^2 [ {\mathbb {R}}^3 ])^3$ 
onto ${{\mathbb{H}}}[ {\mathbb {R}}^3]$ and define the Stokes operator by:  $ {\bf{Au}} = : -\mathbb{P} \Delta {\bf{u}}$, 
for ${\bf{u}} \in D({\bf{A}}) \subset {\mathbb{H}}^{2}[ {\mathbb {R}}^3]$, the domain of ${\bf{A}}$.      If we apply $\mathbb{P}$ to equation (2.1), with 
${{B}}({\mathbf{u}},{\mathbf{u}}) = \mathbb{P}({\mathbf{u}} \cdot \nabla ){\mathbf{u}}$, we can recast equation (2.1) into the standard form:
\beqn
\begin{gathered}
  \partial _t {\mathbf{u}} =  - \nu {\mathbf{Au}} - {{B}}({\mathbf{u}},{\mathbf{u}}) + \mathbb{P}{\mathbf{f}}(t){\text{ in (}}0,T) \times \R^3 , \hfill \\
  {\text{                              }}{\mathbf{u}}(0,{\mathbf{x}}) = {\mathbf{u}}_0 ({\mathbf{x}}){\text{ in }}\R^3, \hfill \\ 
\end{gathered} 
\eeqn
where the orthogonal complement of ${\Ha} $ relative to $\{{L}^{2}(\R^3)\}^3, \;  \{ {\mathbf{v}}\,:\;{\mathbf{v}} = \nabla q,\;q \in \Ha^1[\R^3] \}$, is used to eliminate the pressure term (see Galdi \cite{GA} or [\cite{SY}, \cite{T1},\cite{T2}]). 
\begin{Def}  We say that a velocity vector field in $\R^3$ is \underline{\rm reasonable} if for  $0 \le t<\iy$, there is a continuous function $m(t)>0$, depending only on $t$ and a constant $M_0$, which may depend on ${\bf u}(0)$ and $f$, such that 
\[
0 < m(t) \leqslant \left\| { {\mathbf{u}}(t)} \right\|_{{\mathbb{H}}} \le M_0.
\] 
\end{Def}
The above definition formalizes the requirement that the fluid has nonzero, but bounded positive definite energy.  However, this condition still allows the velocity to approach zero at infinity in a weaker norm.
\subsection{The Nonlinear Term: A Priori Estimates}
The difficulty in proving the existence and uniqueness of global-in-time strong solutions for equation (2.2) is directly linked to the problem of getting good a priori estimates for the nonlinear term ${{B}}({\mathbf{u}},{\mathbf{u}})$.  Let $\Ha_{sd}$ be the closure of $D({\bf{A}}) \cap SD^2[\R^3]$ in the $SD^2$ norm.  
\begin{thm}   If $\bf A$ is the Stokes operator and  ${\bf u}({\bf x},t) \in D({\bf{A}})$ is a reasonable vector field, then
\begin{enumerate}
\item ${\left\langle { \nu{\bf{A}}{\mathbf{u}},{\mathbf{u}}} \right\rangle _{{\mathbb{H}}_{sd} }} = \nu \left\| { {\mathbf{u}} } \right\|_{ {\mathbb{H}}_{sd} }^2$.
\item For ${\bf u}({\bf x},t) \in {\bf SD}^2 \cap D({\bf{A}})$ and each $t \in [0, \iy)$, there exists a constant $M =M({\bf u}({\bf x},0))>0$, such that
\beqn
\left| {\left\langle {B({\mathbf{u}},{\mathbf{u}}),{\mathbf{u}}} \right\rangle _{{\mathbb{H}}_{sd} }} \right| \le M \left\| { {\mathbf{u}} } \right\|_{ {\mathbb{H}}_{sd} }^3. 
\eeqn
\item
\beqn
\left| {\left\langle {{{B}}({\mathbf{u}},{\mathbf{v}}),{\mathbf{w}}} \right\rangle _{{\mathbb{H}}_{sd}} } \right| \le M \left\| {\bf{u}} \right\|_{{\mathbb{H}}_{sd}} \left\| {\bf{w}} \right\|_{{\mathbb{H}}_{sd}} \left \| {\bf{v}} \right\|_{{\mathbb{H}}_{sd}}. 
\eeqn
\item
\beqn
max\{ \left\| {{{B}}({\mathbf{u}},{\mathbf{v}})} \right\|_{{\mathbb{H}}_{sd}}, \ \left\| {{{B}}({\mathbf{v}},{\mathbf{u}})} \right\|_{{\mathbb{H}}_{sd}} \} \leqslant M \left\| {\mathbf{u}} \right\|_{{\mathbb{H}}_{sd}} \left\| {\mathbf{v}} \right\|_{{\mathbb{H}}_{sd}}. 
\eeqn
\end{enumerate}
\end{thm}
\begin{proof} From the definition of the inner product, we have
\beqa
{\left\langle {\nu {\bf{A}}{\mathbf{u}},{\mathbf{u}}} \right\rangle _{{\mathbb{H}}_{sd} }}=\nu
 \sum\nolimits_{k = 1}^\infty  {t_k } \left[ {\int_{\mathbb{R}^n } {{\mathcal{E}}_k ({\mathbf{x}})\cdot {\bf{A}}{\mathbf{u}}({\mathbf{x}})d{\la_n(\bf{x})}} } \right]\left[ {\int_{\mathbb{R}^n } {{\mathcal{E}}_k ({\mathbf{y}})\cdot {\mathbf{u}}({\mathbf{y}})d{\la_n(\bf{y})}} } \right].  
\eeqa
Using the fact that ${\bf{u}} \in D({\bf A})$, it follows that
\[
\begin{gathered} 
{\int_{\mathbb{R}^n } {{\mathcal{E}}_k ({\mathbf{y}}) \cdot {\partial _{y_j }^2 }{\mathbf{u}}({\mathbf{y}})d{\la_n(\bf{y})}} }= {\int_{\mathbb{R}^n } {\partial _{y_j }^2} {{\mathcal{E}}_k ({\mathbf{y}}) \cdot{\mathbf{u}}({\mathbf{y}})d{\la_n(\bf{y})}} } \hfill \\
= {\int_{I_{i} } {\partial _{y_j }^2} {\left( {\xi _l^i(y_1) ,\xi _l^i(y_2) , \cdots \xi _l^i(y_n) } \right) \cdot{\mathbf{u}}({\mathbf{y}})d{\la_n(\bf{y})}} }=(-i)^2{\int_{\mathbb{R}^n } {{\mathcal{E}}_k ({\mathbf{y}}) \cdot {\mathbf{u}}({\mathbf{y}})d{\la_n(\bf{y})}} }.  \hfill \\
\end{gathered}
\]
Using this in the above equation and summing on $j$, we have (${\bf A }=-\mathbb{P} \De$) 
\[ 
{\int_{\mathbb{R}^n } {{\mathcal{E}}_k ({\mathbf{y}}) \cdot {\bf A }{\mathbf{u}}({\mathbf{y}})d{\la_n(\bf{y})}} } ={\int_{\mathbb{R}^n } {{\mathcal{E}}_k ({\mathbf{y}}) \cdot{\mathbf{u}}({\mathbf{y}})d{\la_n(\bf{y})}} }.
\]
It follows that
\beqa
\begin{gathered}
{\left\langle { {\bf{A}}{\mathbf{u}},{\mathbf{u}}} \right\rangle _{{\mathbb{H}}_{sd} }}\hfill \\
= \sum\nolimits_{k = 1}^\infty  {t_k } \left[ {\int_{\mathbb{R}^n } {{\mathcal{E}}_k ({\mathbf{x}})\cdot {\mathbf{u}}({\mathbf{x}})d{\la_n(\bf{x})}} } \right]\left[ {\int_{\mathbb{R}^n } {{\mathcal{E}}_k ({\mathbf{y}})\cdot {\mathbf{u}}({\mathbf{y}})d{\la_n(\bf{y})}} } \right] \hfill \\
 =\left\| {{\mathbf{u}} } \right\|_{ {\mathbb{H}}_{sd} }^2.  \hfill \\
 \end{gathered}
\eeqa
This proves (1).  To prove (2),  let   $\vec \delta ({\mathbf{x}}) = \left( {\delta (x_1 ), \cdots \delta_k (x_3 )} \right)$ be the $n$-dimensional Dirac delta function and set $\hat{\e}=\lt\|\vec \delta ({\mathbf{x}})\rt\|_{\mathbb{H}_{sd} }$.
We start with  
\[
b({\bf u}, {\bf u}, {\mathcal{E}}_k)= \left| {\left\langle {B({\mathbf{u}},{\mathbf{u}}),{\mathcal{E}}_k} \right\rangle _{{\mathbb{H}}_{sd} }} \right|=
\left| {\int_{\mathbb{R}^3 } {\left( {{\mathbf{u}}({\mathbf{x}}) \cdot \nabla } \right){\mathbf{u}}({\mathbf{x}}) \cdot {\mathcal{E}}_k({\mathbf{x}})d{\la_n(\bf{x})}} } \right|
\]
and integrate by parts, to get 
\[
\left| {\int_{\mathbb{R}^3 } {\left\{ {\sum\nolimits_{i = 1}^3 {u_i ({\mathbf{x}})^2 \mcE_k^i ({\mathbf{x}})d{\la_n(\bf{x})}} } \right\}} } \right| \leqslant \mathop {\sup }\limits_k \left\| {\mcE_k } \right\|_\infty  \left\| {\mathbf{u}} \right\|_\Ha^2  \le  \left\| {\mathbf{u}} \right\|_\Ha^2 .
\]
Since ${\bf{u}}$ is reasonable, there is a constant  $\bar{M}$  depending on ${\bf{u}}(0)$ and $f$,  such that  $\left\| {\mathbf{u}} \right\|_2^2 \le \bar{M} \left\| {\mathbf{u}} \right\|_{\mathbb{H}_{sd} }^2$.   We now have
\[
\begin{gathered}
  \left| {\left\langle {B({\mathbf{u}},{\mathbf{u}}),{\mathbf{u}}} \right\rangle _{\mathbb{H}_{sd} } } \right|  \hfill \\
  = \left| {\sum\nolimits_{k= 1}^\infty  {t_k } \left[ {\int_{\mathbb{R}^3 } {\left( {{\mathbf{u}}({\mathbf{x}}) \cdot \nabla } \right){\mathbf{u}}({\mathbf{x}}) \cdot \mcE_k ({\mathbf{x}})d{\la_n(\bf{x})}} } \right]\left[ {\int_{\mathbb{R}^3 } {{\mathbf{u}}({\mathbf{y}}) \cdot \mcE_k ({\mathbf{y}})d{\la_n(\bf{y})}} } \right]} \right| \hfill \\
   \leqslant \bar{M} \hat{\e}^{-2} \left\| {\mathbf{u}} \right\|_{\mathbb{H}_{sd} }^2 \left| {\sum\nolimits_{k= 1}^\infty  {t_k } \left[ {\int_{\mathbb{R}^3 } {{\vec \delta({\mathbf{x}})} \cdot \mcE_k ({\mathbf{x}})d{\la_n(\bf{x})}} } \right]\left[ {\int_{\mathbb{R}^3 } {{\mathbf{u}}({\mathbf{y}}) \cdot \mcE_k ({\mathbf{y}})d{\la_n(\bf{y})}} } \right]} \right| \hfill \\
   \leqslant {M}\left\| {\mathbf{u}} \right\|_{\mathbb{H}_{sd} }^3, \hfill \\ 
\end{gathered} 
\]
where $M=\bar{M} \hat{\e}^{-1}$ and the third line above follows from Schwartz's inequality.  The proofs of (3) and (4) are easy.
\end{proof}
To compare our results, a typical bound available in the $\Ha$ (or energy) norm for equation (2.5) can be found in Sell and You \cite{SY} (see page 366):  
\[
\max \left\{ {{{\left\| {B({\mathbf{u}},{\mathbf{v}})} \right\|}_\mathbb{H}},\;{{\left\| {B({\mathbf{v}},{\mathbf{u}})} \right\|}_\mathbb{H}}} \right\} \leqslant {C_0}{\left\| {{{\mathbf{A}}^{5/8}}{\mathbf{u}}} \right\|_\mathbb{H}}{\left\| {{{\mathbf{A}}^{5/8}}{\mathbf{v}}} \right\|_\mathbb{H}}.
\]
\section{Conclusion}
We have constructed a new class of separable Banach spaces, $S{D^p}[{\mathbb{R}^n}],\;1 \leqslant p \leqslant \infty$, which contain each $L^p$-space as a dense continuous and compact embedding.  These spaces have the remarkable property that, for any multi-index $\al, \; \left\| {{D^\alpha }{\mathbf{u}}} \right\|_{SD} = \left\| {\mathbf{u}} \right\|_{SD}$.  We have shown that our spaces contain the nonabsolutely integrable functions and the space of test functions ${\mathcal{D}}[{\mathbb{R}^n}]$, as a dense continuous embedding.  We have discussed their basic properties and their relationship to ${\mathcal{D'}}[{\mathbb{R}^n}]$, $\mcS[{\mathbb{R}^n}]$ and $\mcS'[{\mathbb{R}^n}]$. As an application, we have obtained new bounds for the nonlinear term of the Navier-Stokes equation. 

\bibliographystyle{amsalpha}

\end{document}